\documentclass[a4paper,12pt,dvipdfmx]{amsart}
\usepackage[hiresbb]{graphicx}
\usepackage{amsmath,amssymb,amsfonts,amsthm,color}
\newtheorem{definition}{Definition}[section]
\newtheorem{theorem}[definition]{Theorem}
\newtheorem{lemma}[definition]{Lemma}
\newtheorem{corollary}[definition]{Corollary}
\newtheorem{remark}[definition]{Remark}
\newtheorem{example}[definition]{Example}
\newtheorem{proposition}[definition]{Proposition}

\begin{document}

\title[Non-linear traces on matrix algebras]{Non-linear traces on matrix algebras, majorization, unitary invariant norms and 
2-positivity}

\author{Masaru Nagisa}
\address[Masaru Nagisa]{Department of Mathematics and Informatics, Faculty of Science, Chiba University, 
Chiba, 263-8522,  Japan: \ Department of Mathematical Sciences, Ritsumeikan University, Kusatsu, Shiga, 525-8577,  Japan}
\email{nagisa@math.s.chiba-u.ac.jp}
\author{Yasuo Watatani}
\address[Yasuo Watatani]{Department of Mathematical Sciences,
Kyushu University, Motooka, Fukuoka, 819-0395, Japan}
\email{watatani@math.kyushu-u.ac.jp}

\maketitle

\begin{abstract}
We study non-linear traces of Choquet type and Sugeno type on matrix algebras. 
They have certain partial additivities. We show that these partial additivities characterize 
non-linear traces of both Choquet type and Sugeno type respectively. 
There exists a close relation among non-linear traces of Choquet type, majorization, unitary invariant norms 
and 2-positivity.

\medskip\par\noindent
AMS subject classification: Primary 46L07, Secondary 47A64

\medskip\par\noindent
Key words: non-linear trace, monotone map

\end{abstract}

\section{Introduction}
In \cite{nagisawatatani} we study several classes of general non-linear positive maps between $C^*$-algebras. 
Recall that Ando-Choi \cite{A-C} and Arveson \cite{Ar2} initiated the study  non-linear completely 
positive maps and extend the Stinespring dilation theorem. 
 Hiai-Nakamura \cite{H-N} studied a non-linear counterpart of Arveson's 
Hahn-Banach type extension theorem \cite {Ar1} for completely positive linear maps. 
Bel\c{t}it\u{a}-Neeb \cite{B-N} studied non-linear completely positive maps and dilation 
theorems for real involutive algebras.  
Recently Dadkhah-Moslehian \cite{D-M} studied some properties of non-linear positive maps 
like Lieb maps and the multiplicative domain for 3-positive maps. Dadkhah-Moslehian-Kian 
\cite{D-M-K} investigate continuity of non-linear positive maps between $C^*$-algebras.

A typical example of non-linear positive mas is given as the 
functional calculus by a continuous positive function.  
See, for example, \cite{bhatia1}, \cite{bhatia2}, \cite{D} and \cite{Si}. 
In particular operator monotone functions are important to study operator means 
in Kubo-Ando theory in \cite{kuboando}.  

One important motivation of non-linear positive maps on $C^*$-algebras is
non-additive measure theory, which was initiated by Sugeno \cite{Su} and  Dobrakov \cite{dobrakov}. 
Choquet integrals \cite{Ch} and Sugeno integrals \cite{Su} are commonly used as non-linear integrals in 
non-additive measure theory. The differences of them are two operations used: 
sum and product for Choquet integrals and maximum and minimum for Sugeno integrals. 
They have partial additivities. 
Choquet integrals have  comonotonic additivity and 
Sugeno integrals have comonotonic F-additivity (fuzzy additivity). Conversely it is known that 
Choquet integrals and Sugeno integrals are characterized by these partial additivities,  see, for example, \cite{De}, 
\cite{schmeidler}, \cite{C-B}, \cite{C-M-R}, \cite{D-G}. 
More precisely, 
comonotonic additivity, positive homogeneity and monotony characterize 
Choquet integrals. Similarly comonotonic F-additivity, F-homogeneity and monotony
characterize Sugeno integrals.
As a general notion of non-linear integrals, we remark the inclusion-exclusion integral by Honda-Okazaki \cite{H-O}
for non-additive monotone measures.

In this paper we study non-linear traces of Choquet type and Sugeno type on matrix algebras, which are 
matrix versions of Choquet integrals and Sugeno integrals.  We 
characterize  non-linear traces of Choquet type and Sugeno type by partial additivities. 
We also show that there exists a close relation among 
non-linear traces of Choquet type, majorization, unitary invariant norms and 2-positivity.

Non-linear traces of Choquet type have comonotonic additivity and 
non-linear traces of Sugeno type have comonotonic F-additivity (fuzzy additivity). Conversely 
these partial additivities characterize 
non-linear traces of both Choquet type and Sugeno type respectively. More precisely
comonotonic additivity, positive homogeneity, unitary invariance and monotony characterize 
non-linear traces of Choquet type. Similarly comonotonic F-additivity, F-homogeneity, unitary invariance 
and monotony characterize non-linear traces of Sugeno type.

Let $A$ and $B$ be $C^*$-algbras.  We denote by $A^+$ and $B^+$ their positive cones. 
 Consider a non-linear positive map 
$\varphi : A^+ \rightarrow B^+$.  We should note that there exist at least three ways to extend such map 
$\varphi$ to $\tilde{\varphi} : A \rightarrow B$:   For $a \in A$, 
\begin{enumerate}
 \item[(1)] $\tilde{\varphi}(a) = \varphi(a_1) - \varphi(a_2)  + i (\varphi(a_3)  - \varphi(a_4))$ for $a = a_1 -a_2 + i(a_3 -a_4)$ 
with $a_1, a_2, a_3, a_4 \in A^+$ and $a_1a_2 = a_3a_4 = 0$.  
 \item[(2)] analytic extension in functional calculus. 
 \item[(3)] $\tilde{\varphi}(a) = \varphi(|a|)$ .
\end{enumerate}
We should choose an appropriate way in various situations. 

\vspace{3mm}

This work was supported by JSPS KAKENHI Grant Number \linebreak
JP17K18739.


\section{non-linear traces of Choquet type}
In this section we study non-linear traces of Choquet type on matrix algebras.

\begin{definition} \rm
Let $\Omega$ be a set and ${\mathcal B}$ a $\sigma$-field on $\Omega$. 
A function $\mu: {\mathcal B} \rightarrow [0, \infty]$ is  called a 
{\it monotone measure} if $\mu$ satisfies 
\begin{enumerate}
\item[$(1)$]  $\mu(\emptyset) = 0$, and 
\item[$(2)$] For any $A,B \in {\mathcal B}$, if $A \subset B$, 
then $\mu(A) \leq \mu(B)$. 

\end{enumerate}
\end{definition}

We recall  the discrete Choquet integral with respect to a monotone measure on a finite set 
$\Omega  = \{1,2, \dots, n\}$.  Let ${\mathcal B} = P(\Omega)$ be the set of all 
subsets of $\Omega$ and $\mu:  {\mathcal B} \rightarrow [0, \infty)$ be a finite 
monotone measure . 

\begin{definition} \rm 
The  discrete Choquet integral  of $f = (x_1, x_2,\dots, x_n) \in [0,\infty)^n$ 
with respect to a monotone measure $\mu$ on a finite set 
$\Omega  = \{1,2, \dots, n\}$ is defined as follows:  
$$
{\rm (C)}\int f d\mu = \sum_{i=1}^{n-1} (x_{\sigma(i) }- x_{\sigma(i+1)})\mu(A_i) 
+ x_{\sigma(n) }\mu(A_n) , 
$$
where $\sigma$ is a permutation on $\Omega$ such that  
$x_{\sigma(1)} \geq x_{\sigma(2)} \geq  \dots \geq x_{\sigma(n)}$ 
and  $A_i = \{\sigma(1),\sigma(2),\dots,\sigma(i)\}$.  
Here we should note that 
$$
f = \sum_{i=1}^{n-1} (x_{\sigma(i) }- x_{\sigma(i+1)})\chi_{A_i} 
+ x_{\sigma(n) }\chi_{A_n} .
$$
\end{definition}

Let $A = {\mathbb C}^n$ and define 
$({\rm C-}\varphi)_{\mu} :( {\mathbb C}^n)^+ \rightarrow {\mathbb C}^+$ 
by the  Choquet integral $({\rm C-}\varphi)_{\mu}(f) = {\rm (C)}\int f d\mu$.  Then 
$({\rm C-}\varphi)_{\mu}$ is a non-linear positive map and 
it is {\it monotone} in the sense that, if $0 \leq f \leq g$ then 
${\rm (C)}\int f d\mu \leq {\rm (C)}\int g d\mu$. It is {\it positively homogeneous} in the sense that 
${\rm (C)}\int kf d\mu = k  \ {\rm (C)}\int f d\mu$ for a scalar $k \geq 0$.

Real valued functions $f$ and $g$ on a set $\Omega$ are said to be {\it comonotonic} if 
$(f(s) -f(t))(g(s)-g(t)) \geq 0$ for any $s,t \in \Omega$, that is, $f(s) < f(t)$ implies$g(s) \leq g(t)$ 
for any $s,t \in \Omega$. In particular consider when  $\Omega = \{x_1, x_2, \dots,x_n\}$ is a finite $n$ points set. 
Then $f$ and $g$ are comonotonic if and only if there exists a permutation $\sigma$ on $\{1,2,\dots,n\}$ such that 
$$
f(x_{\sigma(1)}) \geq f(x_{\sigma(2)}) \geq \dots \geq  f(x_{\sigma(n)}).
$$
$$
g(x_{\sigma(1)}) \geq g(x_{\sigma(2)}) \geq \dots \geq  g(x_{\sigma(n)}).
$$
We should note that comonotonic relation is not transitive and not an equivalent relation. 
The Choquet integral has {\it comonotonic additivity}, that is, 
for any comonotonic pair $f$ and $g$, 
$$
{\rm (C)}\int (f  + g )d\mu= {\rm (C)}\int f d\mu  + {\rm (C)}\int g d\mu
$$
The Choquet integral has comonotonic additivity, positive homogeneity and monotony. Conversely 
 comonotonic additivity, positive homogeneity and monotony  characterize the Choquet integral. 

We shall consider a matrix version of the discrete Choquet integral and its characterization. 
We introduce non-linear traces of Choquet type on matrix algebras. 

\begin{definition} \rm A non-linear positive map 
$\varphi : (M_n({\mathbb C}))^+ \rightarrow  {\mathbb C}^+$ is called a {\it trace} if 
$\varphi$ is unitarily invariant, that is, 
$\varphi(uau^*) = \varphi(a)$ for any $a \in (M_n({\mathbb C}))^+$ and any unitary 
$u \in M_n({\mathbb C})$.  
\begin{itemize}
  \item  $\varphi$ is {\it monotone} if 
$a \leq b$ implies  $\varphi(a) \leq  \varphi(b)$ for any $a,b \in  (M_n({\mathbb C}))^+$. 
  \item  $\varphi$ is {\it positively homogeneous} if 
$\varphi(ka) = k\varphi(a)$ for any $a \in  (M_n({\mathbb C}))^+$ and any  scalar $k \geq 0$.
  \item $\varphi$ is {\it comonotonic additive on the spectrum} if 
$$
\varphi(f(a)  + g(a)) = \varphi(f(a)) + \varphi(g(a)) 
$$  
for any $a \in  (M_n({\mathbb C}))^+$ and 
any comonotonic functions $f$ and $g$   in $C(\sigma(a))$, where 
$f(a)$ is a functional calculus of $a$ by $f$ and  $C(\sigma(a))$ is 
the set of continuous functions on the spectrum $\sigma(a)$ of $a$. 
  \item $\varphi$ is {\it monotonic increasing additive on the spectrum} if 
$$
\varphi(f(a)  + g(a)) = \varphi(f(a)) + \varphi(g(a))
$$  
for any $a \in  (M_n({\mathbb C}))^+$ and 
any monotone increasing functions $f$ and $g$  in $C(\sigma(a))$ . Then by induction, 
we also have 
$$
 \varphi(\sum_{i=1}^n f_i(a)) = \sum_{i=1}^n \varphi (f_i(a))
$$
for any monotone increasing functions $f_1, f_2, \dots, f_n$ in  $C(\sigma(a))$.
\end{itemize}
\end{definition}

A finite monotone measure $\mu:  P(\Omega) \rightarrow [0, \infty)$ on a 
a finite set $\Omega  = \{1,2, \dots, n\}$ is {\it permutation invariant} if 
 $\mu(A) = \mu(\sigma (A))$ for any subset $A \subset \Omega$ and any permutation $\sigma$ on $\Omega$. 
The following lemma is clear.

\begin{lemma} Let $\alpha: \{0,1,2, \dots, n\} \rightarrow [0, \infty)$ be a 
monotone increasing function with $\alpha(0) = 0$, that is, 
$$
0 = \alpha(0) \leq  \alpha(1) \leq  \alpha(2) \leq \dots \leq  \alpha(n).
$$
Define a finite measure $\mu_{\alpha}:  P(\Omega) \rightarrow [0, \infty)$ on 
a finite set $\Omega  = \{1,2, \dots, n\}$ by 
$$
\mu_{\alpha}(A) = \alpha( ^{\#}A) \ \ \ \text{for a subset } A \subset \Omega
$$
where $^{\#}A$ is the cardinality of $A$. Then $\mu_{\alpha}$ is a permutation invariant 
monotone measure. Conversely any  permutation invariant monotone measure 
on a finite set $\Omega  = \{1,2, \dots, n\}$ has this form. 
\end{lemma}

\begin{definition} \rm
Let $\alpha: \{0,1,2, \dots, n\} \rightarrow [0, \infty)$ be a 
monotone increasing function with $\alpha(0) = 0$. We denote by 
$\mu_{\alpha}$ the associated permutation invariant 
monotone measure on $\Omega  = \{1,2, \dots, n\}$. 
Define 
$\varphi_{\alpha} : (M_n({\mathbb C}))^+ \rightarrow  {\mathbb C}^+$ 
as follows:  For $a \in (M_n({\mathbb C}))^+$, \\
let 
$\lambda(a) = (\lambda_1(a),\lambda_2(a),\dots, \lambda_n(a))$ be the list of the 
eigenvalues of $a$ in decreasing order :
$\lambda_1(a) \geq \lambda_2(a) \geq \dots \geq \lambda_n(a)$ with 
counting multiplicities.  Let
\begin{align*}
\varphi_{\alpha}(a) &= \sum_{i=1} ^{n-1} ( \lambda_i(a)- \lambda_{i+1}(a))\mu_{\alpha}(A_i) 
+ \lambda_n (a) \mu_{\alpha}(A_n)\\
&= \sum_{i=1} ^{n-1} ( \lambda_i(a)- \lambda_{i+1}(a)) \alpha( ^{\#}A_i) 
+ \lambda_n (a) \alpha(^{\#}A_n)\\
& =  \sum_{i=1} ^{n-1} ( \lambda_i(a)- \lambda_{i+1}(a)) \alpha(i) 
+ \lambda_n (a) \alpha(n) ,
\end{align*}
where $A_i = \{1,2,\dots,i \}$.
We call $\varphi_{\alpha}$ the non-linear trace of Choquet type associated with $\alpha$.  
Note that $\varphi_{\alpha}$ is norm continuous on $(M_n({\mathbb C}))^+$, since 
each $\lambda_i$ is norm continuous on $(M_n({\mathbb C}))^+$. 
\end{definition}

\begin{remark} \rm
Let  $\mu:  P(\Omega) \rightarrow [0, \infty)$ be a finite monotone measure on 
a finite set $\Omega  = \{1,2, \dots, n\}$, which is not necessarily permutation invariant. 
Put $\alpha (i) = \mu (\{1,2,\dots,i\} )$ and  $\alpha (0) = 0$. Let 
$\mu_{\alpha}$ be the associated permutation invariant monotone measure and 
$\varphi_{\alpha}$ the non-linear trace of Choquet type associated with $\alpha$. 
This $\varphi_{\alpha}$ coincides with the non-linear trace of Choquet type that we introduced in 
our previous paper  \cite{nagisawatatani}.  
\end{remark}

We shall characterize non-linear traces of Choquet type. 

\begin{theorem} 
Let $\varphi : (M_n({\mathbb C}))^+ \rightarrow  {\mathbb C}^+$ be a non-linear 
positive map.  Then the following conditions are equivalent:
\begin{enumerate}
\item[$(1)$]  $\varphi$ is a non-linear trace $\varphi = \varphi_{\alpha}$  of Choquet type associated with  
a monotone increasing function $\alpha: \{0,1,2, \dots, n\} \rightarrow [0, \infty)$  with $\alpha(0) = 0$
\item[$(2)$] $\varphi $ is monotonic increasing additive on the spectrum, unitarily invariant, monotone and 
positively homogeneous. 
\item[$(3)$] $\varphi $ is comonotonic additive on the spectrum, unitarily invariant, monotone and 
positively homogeneous. 
\end{enumerate}
\end{theorem}
\begin{proof}
(1)$\Rightarrow$(2)  Assume that $\varphi$ is a non-linear trace $\varphi = \varphi_{\alpha}$  of Choquet type associated with  
a monotone increasing function $\alpha$. 
For $a \in (M_n({\mathbb C}))^+$, 
let 
$$
a = \sum_{i=1}^n \lambda_i(a) p_i
$$
be the spectral decomposition of $a$, where 
$\lambda(a) = (\lambda_1(a),\lambda_2(a),\dots, \lambda_n(a))$ is the list of the 
eigenvalues of $a$ in decreasing order :
$\lambda_1(a) \geq \lambda_2(a) \geq \dots \geq \lambda_n(a)$ with 
counting multiplicities.  Then the spectrum $\sigma(a) = \{\lambda_i(a) \ | \ i = 1,2,\dots,n\}$. 
For any monotone increasing functions $f$ and $g$  in $C(\sigma(a))$, 
we have the  spectral decompositions 
\begin{gather*}
f(a) = \sum_{i=1}^n f(\lambda_i(a)) p_i, \ \ \ g(a) = \sum_{i=1}^n g(\lambda_i(a)) p_i,   \\
\text{and } (f+g)(a) = \sum_{i=1}^n (f(\lambda_i(a))+ (g(\lambda_i(a))) p_i. 
\end{gather*}
Since $f+g$ is also monotone increasing on $\sigma(a)$ as well as $f$ and $g$, we have 
\begin{gather*}
f(\lambda_1(a)) \geq f(\lambda_2(a)) \geq \dots \geq f(\lambda_n(a)),    \\
g(\lambda_1(a)) \geq g(\lambda_2(a)) \geq \dots \geq g(\lambda_n(a)),   \\
\text{and }  (f+g)(\lambda_1(a)) \geq (f+g)(\lambda_2(a)) \geq \dots \geq (f+g)(\lambda_n(a)).
\end{gather*}
Therefore
\begin{gather*}
\varphi_{\alpha}(f(a)) = \sum_{i=1} ^{n-1} ( f(\lambda_i(a))- f(\lambda_{i+1}(a))) \alpha(i) 
+ f(\lambda_n (a)) \alpha(n),   \\
\varphi_{\alpha}(g(a)) = \sum_{i=1} ^{n-1} ( g(\lambda_i(a))- g(\lambda_{i+1}(a))) \alpha(i) 
+ g(\lambda_n (a)) \alpha(n), 
\end{gather*}
and 
\begin{align*}
  & \varphi_{\alpha}(f(a)+g(a))  = \varphi_{\alpha}((f+g)(a)) \\
 = & \sum_{i=1} ^{n-1} ( (f+g)(\lambda_i(a))- (f+g)(\lambda_{i+1}(a))) \alpha(i) 
+ (f+g)(\lambda_n (a)) \alpha(n) \\
 = & \varphi_{\alpha}(f(a))  + \varphi_{\alpha}(g(a)) . 
\end{align*}
Thus $\varphi_{\alpha}$ is monotonic increasing additive on the spectrum. 

For $a,b \in (M_n({\mathbb C}))^+$, let 
$\lambda(a) = (\lambda_1(a),\lambda_2(a),\dots, \lambda_n(a))$ be  the list of the 
eigenvalues of $a$ in decreasing order and
$\lambda(b) = (\lambda_1(b),\lambda_2(b),\dots,$ $\lambda_n(b))$ be  the list of the 
eigenvalues of $b$ in decreasing order.  
Assume that $a \leq b$. 
By the mini-max  principle for eigenvalues, we have that 
$\lambda_i(a) \leq \lambda_i(b)$ for $i = 1,2,\dots,n$.  Hence 
\begin{align*}
\varphi_{\alpha}(a) &= \sum_{i=1} ^{n-1} ( \lambda_i(a)- \lambda_{i+1}(a)) \alpha(i) 
+ \lambda_n (a) \alpha(n) \\
& = \sum_{i=2} ^{n}  \lambda_i(a)( \alpha(i) -  \alpha(i-1)) +  \lambda_1(a)\alpha(1) \\
& \leq \sum_{i=2} ^{n}  \lambda_i(b)( \alpha(i) -  \alpha(i-1)) +  \lambda_1(b)\alpha(1)
= \varphi_{\alpha}(b). 
\end{align*}
Thus $\varphi_{\alpha}$ is monotone. 

For a positive scalar $k$ , $ka = \sum_{i=1}^n k \lambda_i(a) p_i$ and 
$$
\lambda(ka) = (k\lambda_1(a),k\lambda_2(a),\dots,k \lambda_n(a))
$$ 
is the list of the 
eigenvalues of $ka$ in decreasing order, 
\begin{align*}
\varphi_{\alpha}(ka) &= \sum_{i=1} ^{n-1} ( k \lambda_i(a)- k \lambda_{i+1}(a)) \alpha(i) 
+ k\lambda_n (a) \alpha(n)\\
 &= k (\sum_{i=1} ^{n-1} ( \lambda_i(a)- \lambda_{i+1}(a)) \alpha(i) 
+ \lambda_n (a) \alpha(n) )= k \varphi_{\alpha}(a). 
\end{align*}
Thus $\varphi_{\alpha}$ is positively homogeneous. 
 It is clear 
that $\varphi_{\alpha}$ is unitarily invariant by definition.

(2)$\Rightarrow$(1)  Assume that $\varphi $ is monotonic increasing additive on the spectrum, unitarily invariant, monotone and 
positively homogeneous. Let $I = p_1 + p_2 + \dots +p_n$ be the decomposition of the identity by 
minimal projections. Define  a function $\alpha: \{0,1,2, \dots, n\} \rightarrow [0, \infty)$  with $\alpha(0) = 0$
by $\alpha (i) = \varphi(p_1+ p_2 + \dots + p_i)$. Since  $\varphi $ is unitarily invariant, $\alpha$ does not depend 
on the choice of minimal projections. Since  $\varphi $ is monotone, $\alpha$ is monotone increasing. 
For $a \in (M_n({\mathbb C}))^+$, 
let 
$a = \sum_{i=1}^n \lambda_i(a) p_i$
be the spectral decomposition of $a$, where 
$\lambda(a) = (\lambda_1(a),\lambda_2(a),\dots, \lambda_n(a))$ is the list of the 
eigenvalues of $a$ in decreasing order with counting multiplicities.   
Define $n$ functions $f_1, f_2, \dots, f_n \in C(\sigma(a))$  by 
$$
f_i (x) = \begin{cases}    \lambda_i(a) - \lambda_{i+1}(a), \quad & \text{ if } x\in  \{ \lambda_1(a), \lambda_2(a)\dots, \lambda_i(a) \}  \\
                                 0  &  \text{ if }  x\in \{ \lambda_{i+1}(a), \dots, \lambda_n(a) \}
\end{cases}
$$
for $i = 1,2,\dots, n-1$ and
$$
f_n(x) = \lambda_n(a) \ \ \  \text{ for } x \in \{ \lambda_1(a), \lambda_2(a),\dots, \lambda_n(a) \}.
$$
Then  $f_1, f_2, \dots, f_n$ are monotone increasing functions in  $C(\sigma(a))$ such that 
$$
f_1(x) + f_2(x) + \dots +f_n(x) = x  \ \ \ \text{ for } x \in \{ \lambda_1(a),\lambda_2(a),\dots, \lambda_n(a) \} .
$$
Therefore 
$$
f_1(a) + f_2(a) + \dots +f_n(a) = a .
$$
Moreover we have that 
$$
f_i(a) =  (\lambda_i(a) - \lambda_{i+1}(a))(p_1 + p_2 + \dots +p_i) \quad \text{ for } i=1,2,\dots, n-1
$$
and $ f_n(a) =  \lambda_n(a) (p_1 + p_2 + \dots +p_n) = \lambda_n(a) I$.

Since $\varphi $ is monotonic increasing additive on the spectrum and 
positively homogeneous, we have that 
\begin{align*}
   \varphi(a) & = \varphi(f_1(a) + f_2(a) + \dots +f_n(a)) \\
&=  \varphi(f_1(a)) + \varphi(f_2(a)) + \dots +\varphi(f_n(a)) \\
&=  \sum_{i=1} ^{n-1} \varphi( ( \lambda_i(a)- \lambda_{i+1}(a))(p_1 + p_2 + \dots p_i)) 
+ \varphi(\lambda_n(a) (I))\\
&=  \sum_{i=1} ^{n-1}  ( \lambda_i(a)- \lambda_{i+1}(a))\varphi(p_1 + p_2 + \dots p_i) 
+ \lambda_n(a)\varphi (I)\\
&=  \sum_{i=1} ^{n-1}  ( \lambda_i(a)- \lambda_{i+1}(a))\alpha(i) 
+ \lambda_n(a)\alpha(n) = \varphi_{\alpha}(a).
\end{align*}
Therefore $\varphi$ is equal to the  non-linear trace $\varphi_{\alpha}$  of Choquet type associated with  
a monotone increasing function $\alpha$.

(3)$\Rightarrow$(2)  It is clear from the fact that any monotone increasing functions $f$ and $g$  in $C(\sigma(a))$ 
are comonotonic.

(1)$\Rightarrow$(3)  For $a \in (M_n({\mathbb C}))^+$, 
let $a = \sum_{i=1}^n \lambda_i(a) p_i$
be the spectral decomposition of $a$, where 
$\lambda(a) = (\lambda_1(a),\lambda_2(a),\dots, \lambda_n(a))$ is the list of the 
eigenvalues of $a$ in decreasing order with counting multiplicities.   
For any comonotonic functions $f$ and $g$  in $C(\sigma(a))$, there exists a permutation 
$\tau$ on $\{1,2,\dots,n\}$ such that 
$$
f(\lambda_{\tau(1)}(a)) \geq f(\lambda_{\tau(2)}(a)) \geq \dots \geq f(\lambda_{\tau(n)}(a)), 
$$
$$
g(\lambda_{\tau(1)}(a)) \geq g(\lambda_{\tau(2)}(a)) \geq \dots \geq g(\lambda_{\tau(n)}(a))
$$
by ordering with multiplicities. 
Considering this fact, we can prove  (1)$\Rightarrow$(3) similarly as (1)$\Rightarrow$(2). 
\end{proof}

\section{non-linear traces of Choquet type, majorization, unitary invariant norms and 2-positivity}

In this section we discuss relations among non-linear traces of Choquet type,  the majorization theory 
for eigenvalues and singular values of matrices and unitary invariant norms of matrices. We also studty 
2-positivity of non-linear functional. 
For $a \in (M_n({\mathbb C}))^+$, 
let 
$a = \sum_{i=1}^n \lambda_i(a) p_i$
be the spectral decomposition of $a$, where 
$\lambda(a) = (\lambda_1(a),\lambda_2(a),\dots, \lambda_n(a))$ is the list of the 
eigenvalues of $a$ in decreasing order with counting multiplicities.   
For fixed $i = 1,2,\dots,n$,  we denote by $\lambda_i$ a non-linear map 
$ \lambda_i: (M_n({\mathbb C}))^+ \rightarrow  {\mathbb C}^+$ given by 
$\lambda_i(a)$ for $a \in M_n({\mathbb C})^+$.

\begin{proposition} 
The set \rm{Ch-T} $:= \{ \varphi_{\alpha} | \alpha: \{0,1,2, \dots, n\} \rightarrow [0, \infty)$ is monotone  increasing and  $\alpha(0) = 0\}$ 
of non-linear traces of Choquet type on $(M_n({\mathbb C}))^+$ is a convex subset of 
the affine set $\{ h: (M_n({\mathbb C}))^+ \rightarrow  {\mathbb C}^+\}$ with 
pointwise sum and  positive scalar multiplication. 
Then  {\rm Ch-T} is equal to a convex cone 
$$
\{ \sum _{i=1}^n  c_i \lambda_i  \ | \    c_i \in [0,\infty), i=1, 2,  \dots, n  \}, 
$$
where $\alpha$ and $c_i$ ($i=1,\dots,n)$  are related by 
$c_1 = \alpha(1)$, $c_i = \alpha(i) - \alpha(i-1)$  $(i=2,3,\ldots,n)$  or  $\alpha(j)= \sum_{i=1}^j c_i$  $(j=1,2,\ldots,n)$. 

Moreover consider the normalized set 
$$ 
S :=\{ \varphi \in \text{\rm Ch-T} \; | \varphi (I) = 1 \}. 
$$ 
Then the set of extreme points of the set $S$ is equal to  $\{ \lambda_i  | i=1, 2,  \dots, n  \}$. 
\end{proposition}
\begin{proof}
The proof is based on the following observation:
\begin{align*}
\varphi_{\alpha}(a) &= \sum_{i=1} ^{n-1} ( \lambda_i(a)- \lambda_{i+1}(a)) \alpha(i) 
+ \lambda_n (a) \alpha(n) \\
& = \sum_{i=2} ^{n}  \lambda_i(a)( \alpha(i) -  \alpha(i-1)) +  \lambda_1(a)\alpha(1) \\
& = \sum_{i=2} ^{n}  c_i \lambda_i(a)  + c_1 \lambda_1(a) .
\end{align*}
Let ${\mathbb R}^n_+ = \{ x = (x_1,x_2,\dots,x_n) \in {\mathbb R}^n  |  x_i \geq 0, \; i = 1, 2, \dots, n \}$ be 
the natural convex cone.  
Then the convex cone $\{ \sum _{i=1}^n  c_i \lambda_i  \ | \    c_i \in [0,\infty), i=1, 2, \dots, n  \}$ 
is affine isomorphic to ${\mathbb R}^n_+$  through  $\varphi = \sum _{i=1}^n  c_i \lambda_i \mapsto (c_1,c_2$ $,\dots,c_n)$.  
Since $\lambda_i(I) =1$ for $ i=1, 2, \dots, n$ , the convex subset $S$  is isomorphic to 
$\{ x = (x_1,x_2,\dots,x_n) \in {\mathbb R}^n_+  |  \sum_{i = 1}^n x_i = 1 \}$. 
Therefore the set of extreme points of $S$ is equal to  $\{ \lambda_i  | i=1, 2,  \dots, n  \}$. 
\end{proof}

\begin{example} \rm  
We sometimes denote $\alpha$ by $\alpha = (\alpha(0), \alpha(1),\alpha(2), \dots, \alpha(n))$. 
\begin{enumerate}
\item[(1)] If $\alpha =(0,1,2,3,\dots,n)$, then $\varphi_{\alpha}(a) = {\rm Tr}(a) = \sum_{i=1}^n \lambda_i(a)$ is the usual linear trace. 
\item[(2)] If $\alpha =(0,1,1,1,\dots,1)$, then $\varphi_{\alpha}(a) = \lambda_1(a)$.
\item[(3)] If $\alpha =(\overbrace{0,\ldots,0}^i, \overbrace{1,\ldots,1}^{n-i+1}$, then $\varphi_{\alpha}(a) = \lambda_i(a)$ $(i=2,3,\ldots,n-1)$.
\item[(4)] If $\alpha =(0,0,0,\dots,0,1)$, then $\varphi_{\alpha}(a) = \lambda_n(a)$.
\end{enumerate}
\end{example}

Let $x = (x_1,x_2,\dots,x_n)$ be a vector in ${\mathbb R}^n$. The decreasing rearrangement of $x$ is 
denoted by 
$x^{\downarrow} = (x^{\downarrow}_1, x^{\downarrow}_2, \dots, x^{\downarrow}_n)$. 
For $x,y \in {\mathbb R}^n$, the weak majorization $x \prec _w  y$ is defined by 
$$
\sum_{i=1}^k x^{\downarrow}_i \leq \sum_{i=1}^k y^{\downarrow}_i \quad ( \text{ for }  1 \leq k \leq n).
$$
The majorization $x \prec y$ is defined by 
$$
\sum_{i=1}^k x^{\downarrow}_i \leq \sum_{i=1}^k y^{\downarrow}_i \quad  ( \text{ for }  1 \leq k \leq n)
\text{ and } 
\sum_{i=1}^n x^{\downarrow}_i = \sum_{i=1}^n y^{\downarrow}_i  .
$$
For $a,b \in (M_n({\mathbb C}))^+$, 
if $a \leq b$, then $\lambda_i(a) \leq \lambda_i(b)$ for any $i = 1, 2, \dots, n$. 
If $\lambda_i(a) \leq \lambda_i(b)$ for any $i = 1, 2, \dots, n$, then $\lambda(a) \prec_w \lambda(b)$.  
If $\lambda(a) \prec  \lambda(b)$,  then $\lambda(a) \prec _w \lambda(b)$.  See, for example,  \cite{bhatia1}, 
\cite{H}  and \cite{hiaipetz} for majorization theory of matrices.

Moreover  we see that $a \leq b$ if and only if $\varphi(a) \leq \varphi(b)$ for any positive linear 
functional.  It is known that  $\lambda(a) \prec  \lambda(b)$ if and only if $a$ is in the convex hull of the 
unitary orbits of $b$.  We shall consider similar facts for the condition that 
$\lambda_i(a) \leq \lambda_i(b)$ for any $i = 1, 2, \dots, n$. 

\begin{proposition} 
For $a,b \in (M_n({\mathbb C}))^+$, the following conditions are equivalent:
\begin{enumerate}
\item[$(1)$]  $\varphi_{\alpha}(a) \leq \varphi_{\alpha}(b)$
for any non-linear trace $\varphi_{\alpha}$  of Choquet type associated with  
all monotone increasing function $\alpha: \{0,1,2, \dots, n\} \rightarrow [0, \infty)$  with $\alpha(0) = 0$.
\item[$(2)$] $\lambda_i(a) \leq \lambda_i(b)$ for any $i = 1,\dots, n$.
\item[$(3)$] there exists a contraction $c \in M_n({\mathbb C})$ such that $a = cbc^*$.
\end{enumerate}
\end{proposition}
\begin{proof}
(1)$\Rightarrow$(2)  It is clear, because there exists 
$\alpha = (0,\dots,0,1,1,\dots,1)$ with $\varphi_{\alpha}(a) = \lambda_i(a) $. 

(2)$\Rightarrow$(1)  Since any  non-linear trace $\varphi_{\alpha}$  of Choquet type 
is positively spanned by $\lambda_i$ for $i = 1,\dots, n$, 
(2) implies that $\varphi_{\alpha}(a) \leq \varphi_{\alpha}(b)$. 

(2)$\Rightarrow$(3)  Assume that $\lambda_i(a) \leq \lambda_i(b)$ for any $i = 1, 2, \dots, n$. 
Then there exist constants $d_i$ with $0 \leq d_i \leq 1$ such that $\lambda_i(a) = d_i\lambda_i(b)d_i$. 
Let $d = {\rm diag}(d_1,d_2,\dots,d_n)$ be a daigonal matrix. By diagonalization, there exist 
unitaries $u$ and $v$ in  $M_n({\mathbb C})$ such that 
$uau^* = {\rm diag}(\lambda_1(a),\lambda_2(a),\dots,$ $ \lambda_n(a))$ 
and
$ vbv^* = {\rm diag}(\lambda_1(b),\lambda_2(b),\dots, \lambda_n(b))$.
We have  $uau^* = dvbv^*d^*$. Then $ c:= u^*dv$ is a contraction and $a = cbc^*$. 

(3)$\Rightarrow$(2)  Assume that there exists a contraction $c \in M_n({\mathbb C})$ such that $a = cbc^*$. 
Then,  for any $i = 1,2, \dots, n$,
$$
\lambda_i(a) \leq \| c \|  \lambda_i(b) \| c^* \| \leq \lambda_i(b). 
$$
\end{proof}

If ${\rm Tr}$ is the usual linear trace, then $||a||_1:= {\rm Tr}(|a|)$ is a unitary invariant norm of 
$a \in M_n({\mathbb C})$.  We shall replace 
the usual linear trace by non-linear traces of Choquet type.  

\begin{definition} \rm
 Let $\varphi = \varphi_{\alpha}$ be  a non-linear trace of Choquet type associated with  
a monotone increasing function $\alpha: \{0,1,2, \dots, n\} \rightarrow [0, \infty)$  with $\alpha(0) = 0$. 
Define $|||a|||_{\alpha}:= \varphi_{\alpha}(|a|)$ for $a \in M_n({\mathbb C})$. Since $\varphi_{\alpha}$ 
is unitarily invariant, $|||uav|||_{\alpha} = |||a|||_{\alpha}$ for any unitaries $u$ and $v$. 
\end{definition}

\begin{proposition}In the above setting, assume that $\alpha(1) > 0$.  
If $||| \ |||_{\alpha}$ is a unitary invariant norm, then $\alpha(i+1) + \alpha(i-1) \leq 2\alpha(i)$ 
for $i = 1,2,\dots, n-1$.
\end{proposition}
\begin{proof}
Let $I = p_1 + p_2 + \dots + p_n$ be a resolution of the identity by minimal projections. 
For  $i = 2,\dots, n-1$, let  
$$
a = p_1 + p_2 + 2p_3 + 2p_4 + \dots + 2p_{i+1}
$$
Then 
\begin{align*}
\alpha(i+1) + \alpha(i-1) &  =  \varphi_{\alpha}(a)
 =||| p_1 + p_2 + 2p_3 + 2p_4 + \dots + 2p_{i+1}|||_{\alpha} \\
 & \leq  
|||p_1+ p_3 + \dots p_{i+1}|||_{\alpha} + |||p_2+ p_3 + \dots p_{i+1}|||_{\alpha} \\
& = \alpha(i) + \alpha(i) = 2\alpha(i).
\end{align*}
For $i = 1$, we have that 
$$
\alpha(2) + \alpha(0) = \alpha(2) = ||| p_1 + p_2|||_{\alpha} \le |||p_1|||_{\alpha} +|||p_2|||_{\alpha} = 2\alpha(1).
$$
\end{proof}

\begin{example} \rm
If $\alpha = (0,1,2,\dots,k,k,\dots,k)$, then $\varphi_{\alpha}(a) =  \sum_{i=1} ^{k}  \lambda_i(a)$. 
Therefore  
$$
|||a|||_{\alpha} = \sum_{i=1} ^{k}  \lambda_i(|a|) = \sum_{i=1} ^{k}  s_i(a)
$$
gives a Ky Fan norm, where  $s_i(a) :=  \lambda_i(|a|)$ is the $i$-th singular value of $a \in M_n({\mathbb C})$. 
\end{example}.

It is known that any unitary invariant norm on matrix algebras gives a 2-positive map 
as in \cite{L}  and \cite[Theorem IX.5.10]{bhatia1}.  
We shall study when 
$|||a|||_{\alpha}$ is a unitary invariant norm for a monotone increasing function $\alpha: \{0,1,2, \dots, n\} \rightarrow [0, \infty)$  with $\alpha(0) = 0$. 

The following lemma is known for matrix algebras and the algebra $B(H)$  of all bounded linear operators on a Hilbert space $H$, 
for example, see \cite{Foias-Frazho}       , 
\begin{lemma} 
Let $A$ be a von Neumann algebra. For operators $a,b,c \in A$, 
the following conditions are equivalent:
\begin{enumerate}
\item[$(1)$]  the $2 \times  2$ operator matrix
$\begin{pmatrix}
     a & c\\
     c^* & b
\end{pmatrix}  \geq 0. $
\item[$(2)$] there exists a contraction $k \in A$ such that $c = a^{1/2}kb^{1/2}$.
\end{enumerate}
\end{lemma}
\begin{proof}
The usual proof  gives that we can choose $k$ in the von Neumann algebra $A$. 
\end{proof}

\begin{remark} \rm  The above Lemma 
does not hold for a general $C^*$-algebra $A$. For example, take 
$A = C[0,1]$ be the $C^*$-algebra of the continuous functions on $[0,1]$. Put $a,b,c \in C[0,1]$ by 
 $a(x) =b(x) =  x$  for $x \in [0,1]$ and $c(x) = x|\sin {\frac{1}{x}}|$ for  $x \in (0,1]$ and $c(0) = 0.$
 Then (1) holds but (2) does not hold.  In fact, on the contrary suppose that  
 there were a contraction $k \in C[0,1]$ such that $c = a^{1/2}kb^{1/2}$.  Then 
$k(x) = |\sin {\frac{1}{x}}|$ for  $x \in (0,1]$ and this contradicts that $k$ is continuous on $[0,1]$. 
\end{remark}

The following proposition is known for matrix algebras $M_n({\mathbb C})$ 
and the algebra $B(H)$  of bounded linear operators on a Hilbert space $H$ 
as in \cite{L}  and \cite[Theorem IX.5.10]{bhatia1}.  Using the Lemma above, we can prove it for a general von Neumann algebra essentially by the same proof as follows:

\begin{proposition} 
\label{prop:2-positive-W*}
Let $\varphi : A \rightarrow  {\mathbb C}$ be a non-linear functional on a von Neumann algebra A.  
Then the followings are equivalent:
\begin{enumerate}
\item[$(1)$]  $\varphi$ is 2-positive, that is, 
$\begin{pmatrix}
     \varphi(a) & \varphi(c)\\
     \varphi(c^*) & \varphi(b)
\end{pmatrix} \geq 0$.
for any $a,b,c \in A$ with 
$\begin{pmatrix}
     a & c\\
     c^* & b
\end{pmatrix}
\geq 0$.
\item[$(2)$] $\varphi $ satisfies the following four conditions {\rm (i), (ii), (iii)} and {\rm (iv)}: 
  \begin{enumerate}
    \item[(i)] $\varphi $ is positive, that is, $\varphi(a) \geq 0$ for any $a \in A^+$. 
    \item[(ii)] $\varphi $ is monotone on $A^+$, that is, if $a \leq b$, then 
$\varphi(a) \leq \varphi(b)$ for any $a,b \in  A^+$. 
    \item[(iii)] $\varphi $ is $*$-preserving, that is, $\varphi(a^*) =  \varphi(a)^*$ for any $a \in A$. 
    \item[(iv)] $\varphi $ satisfies Schwartz inequality, that is, 
\[  |\varphi(a^*b) |^2 \leq \varphi(a^*a) \varphi(b^*b)  \text{ for any } a,b \in A.  \] 
   \end{enumerate}
\end{enumerate}
\end{proposition}
\begin{proof} 
(1)$\Rightarrow$(2)  For $a,b \in A$, $0 \leq a \leq b$ if and only if 
$\begin{pmatrix}   
    a & a\\
     a & b
\end{pmatrix}
\geq 0
$. 
Therefore (i) and (ii) is trivial.  
Since for any $a \in A$, 
$\begin{pmatrix}
     I & a\\
     a^*& a^*a
\end{pmatrix}
\geq 0$,  we have  
$\begin{pmatrix}
    \varphi(I) & \varphi(a)\\
     \varphi(a^*)& \varphi(a^*a)
\end{pmatrix}
\geq 0$. Thus $\varphi(a^*) =  \varphi(a)^*$ and (iii) holds. \\
For $a,b \in A$, 
$$
\begin{pmatrix}
     a^*a & a^*b\\
     b^*a & b^*b
\end{pmatrix}
= \begin{pmatrix}
     a^* & 0\\
     b^*& 0
\end{pmatrix}
\begin{pmatrix}
     a & b\\
     0 & 0
\end{pmatrix}
\geq 0. 
$$
Since its determinant is positive, (iv)$|\varphi(a^*b) |^2 \leq \varphi(a^*a) \varphi(b^*b)$ is proved. 

(2)$\Rightarrow$(1) Suppose that 
$\begin{pmatrix}
     a & c\\
     c^* & b
\end{pmatrix}
\geq 0
$
for $a,b,c \in A$. Then there exists a contraction $k \in A$ such that $c = a^{1/2}kb^{1/2}$.
By (i), $\varphi(a) \geq 0$ and $ \varphi(b) \geq 0$. By (iii), $\varphi(c^*) = \varphi(c)^*$. 
By (vi) and (ii), we have 
$$
|\varphi(c)|^2  = |\varphi(a^{1/2}kb^{1/2})|^2 \leq 
\varphi(a^{1/2}a^{1/2})\varphi(b^{1/2}k^*kb^{1/2})\leq \varphi(a) \varphi(b). 
$$
Therefore 
$\begin{pmatrix}
    \varphi(a) & \varphi(c) \\
     \varphi(c^*)& \varphi(b)
\end{pmatrix}
\geq 0
$.
\end{proof}

For a general $C^*$-algebra, we need an additional assumptions for the moment.  We do not know whether
we need the operator norm  continuity of $\varphi$.

\begin{proposition} 
\label{prop:2-positive-C*}
Let $\varphi : A \rightarrow  {\mathbb C}$ be a non-linear functional on a unital $C^*$-algebra $A$. 
Assume that $\varphi$ is operator norm continuous. 
Then the followings are equivalent:
\begin{enumerate}
\item[$(1)$]  $\varphi$ is 2-positive.
\item[$(2)$] $\varphi $ satisfies the following conditions: 
  \begin{enumerate}
  \item[(i)]  $\varphi $ is positive, that is, $\varphi(a) \geq 0$ for any $a \in A^+$. 
  \item[(ii)] $\varphi $ is monotone on $A^+$, that is, if $a \leq b$, then 
$\varphi(a) \leq \varphi(b)$ for any $a,b \in  A^+$. 
  \item[(iii)]  $\varphi $ is $*$-preserving, that is, $\varphi(a^*) =  \varphi(a)^*$ for any $a \in A$.  
  \item[(iv)]  $\varphi $ satisfies Schwartz inequality, that is, 
\[   |\varphi(a^*b) |^2 \leq \varphi(a^*a) \varphi(b^*b)   \text{ for any } a,b \in A.    \]
  \end{enumerate}
\end{enumerate}
\end{proposition}
\begin{proof}
The only non-trivial part is (2)$\Rightarrow$(1).  Suppose that 
$\begin{pmatrix}
     a & c\\
     c^* & b
\end{pmatrix}
\geq 0
$
for $a,b,c \in A$.  First assume that $a$ and $b$ are invertible. Then
\begin{align*}
\begin{pmatrix}
     I& a^{-1/2}cb^{-1/2}\\
     b^{-1/2}c^*a^{-1/2} & I
\end{pmatrix}   &   \\
=\begin{pmatrix}
     a^{-1/2} & 0\\
     0& b^{-1/2}
\end{pmatrix}  &
\begin{pmatrix}
     a & c\\
     c^* & b
\end{pmatrix}
\begin{pmatrix}
     a^{-1/2} & 0\\
     0& b^{-1/2}
\end{pmatrix}
\geq 0 .
\end{align*}
Therefore $k:=a^{-1/2}cb^{-1/2}$ is contraction and $c = a^{1/2}kb^{1/2}$. 
By the same proof of Proposition \ref{prop:2-positive-W*} we have that 
$\begin{pmatrix}
    \varphi(a) & \varphi{\rm (c)}\\
     \varphi(c^*)& \varphi(b)
\end{pmatrix}
\geq 0. $
If $a$ and $b$ are not invertible, then replacing them by  $a +\varepsilon I$ and 
$b +\varepsilon I$, the norm continuity of $\varphi$ implies that 
$\begin{pmatrix}
    \varphi(a) & \varphi{\rm (c)}\\
     \varphi(c^*)& \varphi(b)
\end{pmatrix}
\geq 0. $
\end{proof}

\begin{corollary} 
\label{cor:unitary-invariant-norm-2-positive}
Let $||| \ ||| : M_n({\mathbb C}) \rightarrow  {\mathbb C}$ be a unitary invariant norm on $M_n({\mathbb C})$. 
Then $||| \ |||$ is 2-positive. 
\end{corollary}
\begin{proof}
It is known that any unitary invariant norm $||| \ |||$  on $M_n({\mathbb C})$ satisfies the condition (2) in 
Proposition \ref{prop:2-positive-W*} as in \cite[IX.5]{bhatia1}.  Therefore $||| \ |||$ is 2-positive. 
\end{proof}

For example, the Ky Fan $k$-norms defined as $|| a ||_{(k)} := s_1(a) + s_2(a) + \dots + s_k(a)$ is 2-positive 
and $s_1$ is also 2-positive but 
$s_2 : M_3({\mathbb C}) \rightarrow  {\mathbb C}$  is not 2-positive.  In fact, consider the diagonal 
matrices $a ={\rm diag}(1,1,3), b = {\rm diag}(3,1,1), c = {\rm diag}(\sqrt{3},1,\sqrt{3})$ in $M_3({\mathbb C})$. 
Then $
\begin{pmatrix}
     a & c\\
     c^* & b
\end{pmatrix}
\geq 0 $ and 
$$\begin{pmatrix}
    s_2(a) & s_2(c) \\
     s_2(c^*)& s_2(b)
\end{pmatrix}
=  \begin{pmatrix}
     1 & \sqrt{3}\\
     \sqrt{3} & 1
\end{pmatrix}
$$
is not positive.  We shall extend this example that $s_2$  is not 2-positive as follows:

\begin{proposition} 
\label{prop:not-2-positive}
Let   $\varphi = \sum _{i=1}^n  c_i \lambda_i$ with $c_i \in [0,\infty), i=1, \dots, n $.  Define 
 $s(a)= \varphi(|a|)$ for $a \in M_n({\mathbb C})$.  If $c_k <c_{k+1}$ for some $k$, then 
 $s$ is not 2-positive. 
\end{proposition}
\begin{proof}
Suppose that  $c_k <c_{k+1}$.
We can choose a number $t>1$ such that $c_{k+1} > \sqrt{t} c_k$.  
Then it holds $tc_k + c_{k+1} < \sqrt{t}(c_k + c_{k+1})$.
Consider the following positive diagonal 
matrices 
$$
a ={\rm diag}(a_1,a_2,\dots,a_n), b ={\rm diag}(b_1,b_2,\dots,b_n)
$$
$$
c ={\rm diag}(c_1,c_2,\dots,c_n)
$$ 
 in  
$M_n({\mathbb C})$ such that 
$$
a_i = b_i = c_i = t  \ \text{for}  \ i = 1,\dots ,k-1,  
$$
$$
a_i = b_i = c_i = 1 \ \text{for} \ i = k+2,\dots ,n
$$ 
and 
$$
a_k = t, a_{k+1} = 1, \ b_k = 1, b_{k+1} = t, \ c_k = \sqrt{t}, c_{k+1} = \sqrt{t}.
$$
Define 2 by 2 operator matrix $A \in M_2(M_n({\mathbb C}))$ by 
$$
A = \begin{pmatrix}
     a & c\\
     c & b
\end{pmatrix}.
$$
Since $c = \sqrt{a} I \sqrt{b}$, $A$ is positive.  Consider 
$$
B = \begin{pmatrix}
    s(a) & s(c) \\
    s(c) & s(b)
\end{pmatrix}
=  \begin{pmatrix}
     \sum _{i=1}^n  c_i \lambda_i(a) & \sum _{i=1}^n  c_i \lambda_i(c)\\
     \sum _{i=1}^n  c_i \lambda_i(c) & \sum _{i=1}^n  c_i \lambda_i(b)
\end{pmatrix}.
$$
Then we have
\begin{align*}
  s(a)=s(b) & = t\sum_{i=1}^k c_i + \sum_{i=k+1}^n c_i = t\sum_{i=1}^{k-1} c_i +tc_k + c_{k+1} + \sum_{i=k+2}^n c_i  \\
       & < t\sum_{i=1}^{k-1} c_i +\sqrt{t}(c_k + c_{k+1}) + \sum_{i=k+2}^n c_i = s(c).
\end{align*} 
Since $\det  B = s(a)s(b) - s(c)^2 < 0$, this shows that $s$ is not 2-positive. 
\end{proof}

\begin{corollary} 
Let $||| \ ||| : A \rightarrow  {\mathbb C}$ be an semi-norm on a general $C^*$-algebra $A$. 
If $\varphi  = ||| \ |||$ satisfies the condition (2) in Proposition \ref{prop:2-positive-C*}, 
then $||| \ |||$ is 2-positive without assuming that $||| \ |||$ is operator norm continuous. 
\end{corollary}
\begin{proof}
Without assuming that $||| \ |||$ is operator norm continuous, it is enough to note that 
$$
|||a +\varepsilon I||| - |||a||| \leq |||\varepsilon I||| = \varepsilon|||I||| \rightarrow 0. 
$$
\end{proof}

\begin{theorem} 
Let $\varphi = \varphi_{\alpha}$ be  a non-linear trace of Choquet type associated with  
a monotone increasing function $\alpha: \{0,1,2, \dots, n\} \rightarrow [0, \infty)$  with $\alpha(0) = 0$ and 
$\alpha(1) > 0$.  Put $c_i := \alpha(i) - \alpha(i - 1)$ for $i =1,2,\dots,n$. Recall that 
$ \varphi_{\alpha} = \sum _{i=1}^n  c_i \lambda_i$. 
Define $|||a|||_{\alpha}:= \varphi_{\alpha}(|a|)$ for $a \in M_n({\mathbb C})$. 
Then the following conditions are equivalent: 
\begin{enumerate}
\item[$(1)$] $\alpha$ is concave in the sense that 
$\frac{\alpha(i +1) + \alpha(i - 1)}{2} \leq \alpha(i), \;  (i =1,2,$ $\dots,n-1)$.
\item[$(2)$] $(c_i)_i$ is a decreasing sequence: $c_1 \geq c_2 \geq \dots \geq c_n$.
\item[$(3)$] $||| \ |||_{\alpha}$ is a unitary invariant norm on $\in M_n({\mathbb C})$.
\item[$(4)$] $||| \ |||_{\alpha}$ is 2-positive.
\end{enumerate}
\end{theorem}
\begin{proof}
It is trivial that (1) $\Leftrightarrow$ (2). \\
(2) $\Rightarrow$ (3): Suppose that $(c_i)_i$ is a decreasing sequence. 
The Ky Fan $k$-norms is defined as 
$$
|| a ||_{(k)} := s_1(a) + s_1(a) + \dots + s_k(a) = \lambda_1(|a|) + \lambda_2(|a|)+ \dots + \lambda_k(|a|))
$$ 
for $a \in M_n({\mathbb C})$. 
Then 
$$
|||a|||_{\alpha}= \varphi_{\alpha}(|a|) = \sum_{i =1}^n c_i \lambda_i(|a|) = 
c_n || a ||_{(n)} + \sum_{i =1}^{n-1}( c_i  - c_{i+1})|| a ||_{(i)}
$$
Therefore $||| \ |||_{\alpha}$ is a unitary invariant norm on $\in M_n({\mathbb C})$. \\ 
(3) $\Rightarrow$ (4):Use Corollary \ref{cor:unitary-invariant-norm-2-positive}. \\
(4) $\Rightarrow$ (2):Use Proposition \ref{prop:not-2-positive}. 
\end{proof}

\section{non-linear traces of Sugeno type}
In this section we study non-linear traces of Sugeno type on matrix algebras. 
We recall  the Sugeno integral with respect to a monotone measure on a finite set 
$\Omega  = \{1,2, \dots, n\}$.  

\begin{definition} \rm 
The  discrete Sugeno integral  of $f = (x_1, x_2,\dots, x_n) \in [0,\infty)^n$ 
with respect to a monotone measure $\mu$ on a finite set 
$\Omega  = \{1,2, \dots, n\}$ is defined as follows:  
$$
{\rm (S)}\int f d\mu = \vee_{i=1}^{n} (x_{\sigma(i) } \wedge \mu(A_i) ) , 
$$
where $\sigma$ is a permutation on $\Omega$ such that  
$x_{\sigma(1)} \geq x_{\sigma(2)} \geq  \dots \geq x_{\sigma(n)}$, 
$A_i = \{\sigma(1),\sigma(2),\dots,\sigma(i)\}$ and 
$\vee =\max$ , $\wedge = \min$. 
Here we should note that 
$$
 f = \vee_{i=1}^{n} (x_{\sigma(i) } \chi_{A_i}) .
$$
\end{definition}

Let $A = {\mathbb C}^n$ and define 
$(\text{{\rm S-}}\varphi)_{\mu} :( {\mathbb C}^n)^+ \rightarrow {\mathbb C}^+$ 
by the  Sugeno integral $(\text{{\rm S-}}\varphi)_{\mu}(f) ={\rm (S)}\int f d\mu$.  
Then $(\text{{\rm S-}}\varphi)_{\mu}$ is a non-linear monotone positive map.

We shall consider a matrix version of the discrete Sugeno integral. 

\begin{definition} \rm
Let $\alpha: \{0,1,2, \dots, n\} \rightarrow [0, \infty)$ be a 
monotone increasing function with $\alpha(0) = 0$ and 
$\mu_{\alpha}$ be the associated permutation invariant 
monotone measure on $\Omega  = \{1,2, \dots, n\}$. 
Define 
$\psi_{\alpha} : (M_n({\mathbb C}))^+ \rightarrow  {\mathbb C}^+$ 
as follows:  For $a \in (M_n({\mathbb C}))^+$, 
let $\lambda(a) = (\lambda_1(a),\lambda_2(a),$ $\dots, \lambda_n(a))$ be the list of the 
eigenvalues of $a$ in decreasing order :
$\lambda_1(a) \geq \lambda_2(a) \geq \dots \geq \lambda_n(a)$ with 
counting multiplicities.  
Let
\begin{align*}
\psi_{\alpha}(a) &= \lor_{i=1} ^{n} ( \lambda_i(a) \land \mu_{\alpha}(A_i)) \\
&= \lor_{i=1} ^{n} ( \lambda_i(a) \land \alpha( ^{\#}A_i) ) \\
&= \lor_{i=1} ^{n} ( \lambda_i(a) \land \alpha(i)) ,
\end{align*}
where $A_i = \{1,2,\dots,i \}$.
We call $\psi_{\alpha}$ the non-linear trace of Sugeno type associated with $\alpha$. 
\end{definition}

\begin{definition} \rm A non-linear positive map 
Let $\psi : (M_n({\mathbb C}))^+ \rightarrow  {\mathbb C}^+$  be a non-linear positive map.
\begin{itemize}
  \item $\psi$ is {\it positively F-homogeneous} if 
$\psi(kI \land a) = k \land \psi(a)$ for any $a \in  (M_n({\mathbb C}))^+$ and any scalar $k \geq 0$.
  \item $\psi$ is {\it comonotonic F-additive on the spectrum} if 
$$
\psi(f(a)  \lor g(a)) = \psi(f(a)) \lor \psi(g(a)) 
$$  
for any $a \in  (M_n({\mathbb C}))^+$ and 
any comonotonic functions $f$ and $g$   in $C(\sigma(a))$, where 
$f(a)$ is a functional calculus of $a$ by $f$. 
  \item $\psi$ is {\it monotonic increasing F-additive on the spectrum} if 
$$
\psi(f(a)  \lor  g(a)) = \psi(f(a)) \lor  \psi(g(a))
$$  
for any $a \in  (M_n({\mathbb C}))^+$ and 
any monotone increasing functions $f$ and $g$  in $C(\sigma(a))$. 
Then by induction, we also have 
\[   \psi(\bigvee_{i=1}^n f_i(a)) = \bigvee_{i=1}^n \psi(f(a_i))  \]
for any monotone increasing functions $f_1, f_2, \dots, f_n$ in  $C(\sigma(a))$.
\end{itemize}
\end{definition}

We shall characterize non-linear traces of Sugeno type. 
\begin{theorem} 
Let $\psi : (M_n({\mathbb C}))^+ \rightarrow  {\mathbb C}^+$ be a non-linear 
positive map.  Then the following are equivalent:
\begin{enumerate}
\item[$(1)$]  $\psi$ is a non-linear trace $\psi = \psi_{\alpha}$  of Sugeno type associated with  
a monotone increasing function $\alpha: \{0,1,2, \dots, n\} \rightarrow [0, \infty)$  with $\alpha(0) = 0$.
\item[$(2)$] $\psi$ is monotonic increasing F-additive on the spectrum, unitarily invariant, monotone,  
positively F-homogeneous and 
$$\lim_{c\rightarrow \infty} \psi(cI) < +\infty .$$ 
\item[$(3)$] $\psi$ is comonotonic F-dditive on the spectrum, unitarily invariant, monotone and 
positively F-homogeneous and 
$$\lim_{c\rightarrow \infty} \psi(cI) < +\infty .$$  
\end{enumerate}
\end{theorem}
\begin{proof}
(1)$\Rightarrow$(2) Assume that $\psi$ is a non-linear trace $\psi = \psi_{\alpha}$  of Sugeno type associated with  
a monotone increasing function $\alpha$. 
For $a \in (M_n({\mathbb C}))^+$, 
let 
$$
a = \sum_{i=1}^n \lambda_i(a) p_i
$$
be the spectral decomposition of $a$, where 
$\lambda(a) = (\lambda_1(a),\lambda_2(a),\dots, \lambda_n(a))$ is the list of the 
eigenvalues of $a$ in decreasing order :
$\lambda_1(a) \geq \lambda_2(a) \geq \dots \geq \lambda_n(a)$ with 
counting multiplicities.  Then the spectrum $\sigma(a) = \{ \lambda_i(a) \ | \  i =1,2,\dots, n\} $. 
For any monotone increasing functions $f$ and $g$  in $C(\sigma(a))$, 
we have the  spectral decompositions 
$$
f(a) = \sum_{i=1}^n f(\lambda_i(a)) p_i, \ \ \ g(a) = \sum_{i=1}^n g(\lambda_i(a)) p_i, 
$$
$$
\text{and }  (f \lor g)(a) = \sum_{i=1}^n (f(\lambda_i(a)) \lor (g(\lambda_i(a))) p_i. 
$$
Since $f \lor g$ is also monotone increasing on $\sigma(a)$ as well as $f$ and $g$, we have 
$$
f(\lambda_1(a)) \geq f(\lambda_2(a)) \geq \dots \geq f(\lambda_n(a)), 
$$ 
$$
g(\lambda_1(a)) \geq g(\lambda_2(a)) \geq \dots \geq g(\lambda_n(a)), 
$$ 
$$
\text{and }  (f \lor g)(\lambda_1(a)) \geq (f \lor g)(\lambda_2(a)) \geq \dots \geq (f \lor g)(\lambda_n(a)). 
$$ 
Therefore
$$
\psi_{\alpha}(f(a)) = \lor_{i=1} ^{n} ( f(\lambda_i(a)) \land \alpha(i)), 
$$
$$
\psi_{\alpha}(g(a)) = \lor_{i=1} ^{n} ( g(\lambda_i(a)) \land \alpha(i)), 
$$
and 
\begin{align*}
& \psi_{\alpha}((f \lor g)(a))  = \lor_{i=1} ^{n} ( (f \lor g)(\lambda_i(a)) \land \alpha(i))\\
 = & \lor_{i=1} ^{n} (((f (\lambda_i(a)) \lor g (\lambda_i(a)) ) \land \alpha(i))\\
 = & \lor_{i=1} ^{n} (((f (\lambda_i(a)) \land \alpha(i))  \lor (g(\lambda_i(a)) \land \alpha(i)) )  \\
 = & (\lor_{i=1} ^{n} ( (f (\lambda_i(a)) \land \alpha(i)))  \lor ( \lor_{i=1} ^{n} ( (g (\lambda_i(a)) \land \alpha(i)))\\
 = & \psi_{\alpha}(f(a)) \lor  \psi_{\alpha}(g(a)) . 
\end{align*}
Thus $\psi_{\alpha}$ is monotonic increasing F-additive on the spectrum. 

For $a,b \in (M_n({\mathbb C}))^+$, let 
$\lambda(a) = (\lambda_1(a),\lambda_2(a),\dots, \lambda_n(a))$ be  the list of the 
eigenvalues of $a$ in decreasing order and 
$\lambda(b) = (\lambda_1(b),\lambda_2(b),\dots,$  $\lambda_n(b))$ be  the list of the 
eigenvalues of $b$ in decreasing order.  
Assume that $a \leq b$. 
By the mini-max  principle for eigenvalues, we have that 
$\lambda_i(a) \leq \lambda_i(b)$ for $i = 1,2,\dots,n$.  Hence 
\begin{align*}
\psi_{\alpha}(a) & = \lor_{i=1} ^{n} ( \lambda_i(a) \land \alpha(i)) \\
& \leq \lor_{i=1} ^{n} ( \lambda_i(b) \land \alpha(i)) = \psi_{\alpha}(b).
\end{align*}
Thus $\psi_{\alpha}$ is monotone. 

For a positive scalar $k$ , $kI \land a = \sum_{i=1}^n(k \land \lambda_i(a)) p_i$ and 
$\lambda(ka) = (k \land \lambda_1(a),k \land \lambda_2(a),\dots,k  \land \lambda_n(a))$ is the list of the 
eigenvalues of $kI \land a$ in decreasing order, hence 
\begin{align*}
\psi_{\alpha}(k \land a) &= \lor_{i=1} ^{n} ( k \land \lambda_i(a) \land \alpha(i)) \\
 &= k \land (\lor_{i=1} ^{n} (\lambda_i(a) \land \alpha(i)) ) 
 = k  \land \psi_{\alpha}(a). 
\end{align*}
Thus $\psi_{\alpha}$ is positively F-homogeneous. 
It is clear that $\psi_{\alpha}$ is unitarily invariant by definition and
$$
\lim_{c\rightarrow \infty} \psi_{\alpha}(cI) = \lim_{c\rightarrow \infty}( \lor_{i=1} ^{n} (c \land \alpha(i)) ) 
= \lim_{c\rightarrow \infty} \lor_{i=1} ^{n} \alpha(i) = \alpha (n) < +\infty.
$$

(2)$\Rightarrow$(1) 
Assume that $\psi$ is monotonic increasing F-additive on the spectrum, unitarily invariant, monotone,  
positively F-homogeneous and $\lim_{c\rightarrow \infty} \psi(cI) < +\infty$. 
Let $I = p_1 + p_2 + \dots + p_n$ be the decomposition of the identity by 
minimal projections. Define  a function $\alpha: \{0,1,2, \dots, n\} \rightarrow [0, \infty)$  with $\alpha(0) = 0$
by 
$$
\alpha (i) := \lim_{c\rightarrow \infty}\psi(c(p_1+ p_2 + \dots + p_i)) \leq \lim_{c\rightarrow \infty}\psi(cI) < \infty. 
$$
Since $\psi$ is monotone, $\alpha$ is monotone increasing. 
Since  $\psi $ is unitarily invariant, $\alpha$ does not depend 
on the choice of minimal projections.  
For $a \in (M_n({\mathbb C}))^+$, 
let 
$a = \sum_{i=1}^n \lambda_i(a) p_i$
be the spectral decomposition of $a$, where 
$\lambda(a) = (\lambda_1(a),\lambda_2(a),\dots, \lambda_n(a))$ is the list of the 
eigenvalues of $a$ in decreasing order with counting multiplicities.   

Define  functions $f, f_1, f_2, \dots, f_n \in C(\sigma(a))$  by $f(x) =x$ and for $i = 1,2,\dots,n$
$$
f_i = \lambda_i(a) \chi_{\{\lambda_1(a),\lambda_2(a),\dots, \lambda_i(a)\}} = \lambda_i(a)I \land 
c \chi_{\{\lambda_1(a),\lambda_2(a),\dots, \lambda_i(a)\}}
$$
for sufficient large $c \geq \lambda_1(a) = \|a\|$, which does not depend on such $c$. 
Each $f_i$ is monotone increasing function on $\sigma(a)$. Since 
$f = \lor_{i=1} ^{n} f_i$, we have that 
$
a = (\lor_{i=1} ^{n} f_i)(a). 
$
Since $\psi$ is monotonic increasing F-additive on the spectrum and positively F-homogeneous, 
we have that 
\begin{align*}
&  \psi(a) = \psi((\lor_{i=1} ^{n} f_i)(a))  = \lor_{i=1} ^{n}\psi(f_i(a)) \\
= &  \lor_{i=1} ^{n}\psi((\lambda_i(a)I \land 
c \chi_{\{\lambda_1(a),\lambda_2(a),\dots, \lambda_i(a)\}})(a)) \\
= &   \lor_{i=1} ^{n}(\lambda_i(a) \land 
\psi((c \chi_{\{\lambda_1(a),\lambda_2(a),\dots, \lambda_i(a)\}})(a)) \\
= &  \lor_{i=1} ^{n} (\lambda_i(a) \land \psi(c(p_1 + p_2 + \dots p_i)) \\
= &  \lor_{i=1} ^{n} (\lambda_i(a) \land (\lim_{c\to \infty}\psi(c(p_1 + p_2 + \dots p_i)) \\
= &  \lor_{i=1} ^{n} (\lambda_i(a) \land \alpha(i)) = \psi_{\alpha}(a).
\end{align*}
Therefore $\psi$ is a non-linear trace $\psi_{\alpha}$  of Sugeno type associated with  
a monotone increasing function $\alpha$. 

(3)$\Rightarrow$(2)  It is clear from the fact that any monotone increasing functions $f$ and $g$  in $C(\sigma(a))$ 
are comonotonic. 

(1)$\Rightarrow$(3)  For $a \in (M_n({\mathbb C}))^+$, 
let $a = \sum_{i=1}^n \lambda_i(a) p_i$
be the spectral decomposition of $a$, where 
$\lambda(a) = (\lambda_1(a),\lambda_2(a),\dots, \lambda_n(a))$ is the list of the 
eigenvalues of $a$ in decreasing order with counting multiplicities.   
For any comonotonic functions $f$ and $g$  in $C(\sigma(a))$, there exists a permutation 
$\tau$ on $\{1,2,\dots,n\}$ such that 
$$
f(\lambda_{\tau(1)}(a)) \geq f(\lambda_{\tau(2)}(a)) \geq \dots \geq f(\lambda_{\tau(n)}(a)), 
$$
$$
g(\lambda_{\tau(1)}(a)) \geq g(\lambda_{\tau(2)}(a)) \geq \dots \geq g(\lambda_{\tau(n)}(a))
$$
by ordering with multiplicities. 
Considering this fact, we can prove  (1)$\Rightarrow$(3) similarly as (1)$\Rightarrow$(2). 
\end{proof}

\end{document}